\documentclass[10pt, reqno]{amsart}
\usepackage{amssymb}
\usepackage{hyperref}
\usepackage{color}
\usepackage{cite}

\numberwithin{equation}{section}

\newtheorem{theorem}{Theorem}[section]

\newtheorem{lemma}[theorem]{Lemma}

\newtheorem{proposition}[theorem]{Proposition}

\theoremstyle{definition}
\newtheorem{definition}[theorem]{Definition}
\newtheorem{remark}[theorem]{Remark}

\theoremstyle{remark}

\newcommand{\R}{\mathbb{R}}
\newcommand{\C}{\mathbb{C}}

\newcommand{\Z}{\mathbb{Z}}
\newcommand{\A}{\mathcal{A}}
\renewcommand{\P}{\mathbb{P}}

\newcommand{\eps}{\varepsilon}
\newcommand{\wh}{\widehat}

\newcommand{\op}[1]{{\left\vert\kern-0.25ex\left\vert\kern-0.25ex\left\vert #1 
    \right\vert\kern-0.25ex\right\vert\kern-0.25ex\right\vert}}

\newcommand{\qtq}[1]{\quad\text{#1}\quad}

\newcounter{smalllist}

\begin{document}
\title[Mass-subcritical NLS]{Random data final-state problem for the mass-subcritical NLS in $L^2$}
\author{Jason Murphy}
\address{Department of Mathematics, University of California, Berkeley}
\email{murphy@math.berkeley.edu}

\begin{abstract} We study the final-state problem for the mass-subcritical NLS above the Strauss exponent. For $u_+\in L^2$, we perform a physical-space randomization, yielding random final states $u_+^\omega\in L^2$. We show that for almost every $\omega$, there exists a unique, global solution to NLS that scatters to $u_+^\omega$.  This complements the deterministic result of Nakanishi \cite{Nakanishi}, which proved the existence (but not necessarily uniqueness) of solutions scattering to prescribed $L^2$ final states.
\end{abstract}

\maketitle

\section{Introduction}

We consider power-type nonlinear Schr\"odinger equations (NLS) of the form
\begin{equation}\label{nls}
(i\partial_t + \Delta) u = \mu |u|^p u,
\end{equation}
where $u:\R_t\times\R_x^d\to\C$, $p>0$, and $\mu\in\{\pm1\}$.  The case $\mu=1$ is the {defocusing} case, while $\mu=-1$ gives the {focusing} case.

The scaling symmetry 
\begin{equation}\label{scaling}
u(t,x)\mapsto \lambda^{\frac{2}{p}} u(\lambda^2 t,\lambda x)
\end{equation}
defines a notion of \emph{criticality} for \eqref{nls}.  In particular, a space of initial data is critical if its norm is invariant under this rescaling. The \emph{mass-critical NLS} refers to the case $p=\frac4d$.  In this case the conserved \emph{mass} of solutions
\[
M(u(t)):=\int_{\R^d} |u(t,x)|^2 \,dx 
\]
is invariant under rescaling.

We consider the \emph{mass-subcritical} regime, i.e. $0<p<\frac{4}{d}$.  In this case, any initial data $u_0\in L^2$ leads to a local-in-time solution $u$.  As a consequence of mass-subcriticality, the time of existence depends only on the $L^2$-norm of $u_0$.  In particular, by the conservation of mass, solutions are automatically global-in-time.  For further details, we refer the reader to the textbook \cite{Cazenave}.

 The long-time behavior of solutions for the mass-subcritical NLS is a rich and interesting subject.  We briefly review some relevant results, noting that our discussion is far from exhaustive. 

\subsection*{The initial-value problem.} Many results have been established for the initial-value problem when one selects data from the weighted space $\Sigma$, which is defined via the norm
\[
\|f\|_\Sigma^2 = \|f\|_{H^1}^2 + \|xf\|_{L^2}^2. 
\]
Tsutsumi and Yajima \cite{TsutsumiYajima} established that in the defocusing case, for $\tfrac{2}{d}<p<\tfrac4d$ and $u_0\in\Sigma$, the solution $u$ to \eqref{nls} with data $u_0$ \emph{scatters} in $L^2$, that is, there exists $u_+\in L^2$ such that 
\begin{equation}\label{L2-scatter}
\lim_{t\to\infty} \|u(t)-e^{it\Delta} u_+\|_{L^2} = 0.
\end{equation}
This result is sharp in the following sense: for $0<p\leq \tfrac2d$, any solution to \eqref{nls} that satisfies \eqref{L2-scatter} must be identically zero \cite{Barab}.

Restricting still to the defocusing case, Cazenave and Weissler \cite{CazenaveWeissler} proved that for a smaller range of $p$, the solution scatters in $\Sigma$, that is,
\begin{equation}\label{Sigma-scatter}
\lim_{t\to\infty} \|e^{-it\Delta}u(t) - u_+\|_\Sigma = 0.
\end{equation}
In particular, \eqref{Sigma-scatter} holds for all $p>\frac{4}{d+2}$ in the small-data regime, while for arbitrary data the result is restricted to $p\in[p_0(d),\tfrac{4}{d})$, where $p_0(d)$ is the \emph{Strauss exponent} defined by
\begin{equation}\label{Strauss}
p_0(d) = \tfrac{2-d+\sqrt{d^2+12d+4}}{2d}.
\end{equation}
The question of scattering in $\Sigma$ for arbitrary data for $p\in(\tfrac{4}{d+2},p_0(d))$ remains open. 

For the scattering results of \cite{CazenaveWeissler, TsutsumiYajima}, an important role is played by the \emph{pseudoconformal energy estimate}, which implies decay for the potential energy $\|u(t)\|_{L^{p+2}}$ that matches the rate for linear solutions.  For $p>p_0(d)$, this estimate implies critical space-time bounds for solutions that can be used to deduce scattering. 

\subsection*{Existence of wave operators} A counterpart to the initial-value problem is the \emph{final-state problem}: given $u_+$, can one find a (unique) solution $u$ that scatters to $u_+$?  If one can do this, one calls the map $u_+\mapsto u(0)$ the \emph{wave operator}.  For $p\in(\tfrac{4}{d+2},\tfrac{4}{d})$, one has the existence of wave operators in the $\Sigma$ topology \cite{CazenaveWeissler}.  In fact, one can prove results in critical weighted spaces essentially all the way down to $p>\frac2d$ \cite{Nakanishi}.

The result that is most relevant to this note is due to Nakanishi \cite{Nakanishi}.  He showed that in dimensions $d\geq 3$, for any $\tfrac{2}{d}<p<\tfrac{4}{d}$ and any $u_+\in L^2$ there exists a solution $u$ to \eqref{nls} satisfying \eqref{L2-scatter}.  However, his arguments did not yield \emph{uniqueness} of the solution, and thus one cannot conclude existence of wave operators in $L^2$.

Nonetheless, Nakanishi's result is quite remarkable, in the sense that working with merely $L^2$ final states makes the final-state problem essentially a \emph{supercritical} problem.  Indeed, the pseudoconformal transformation 
\[
u(t,x) = (2it)^{-\frac{d}{2}}e^{\frac{i|x|^2}{4t}} \bar{v}(\tfrac{1}{2t},\tfrac{x}{2t})
\]
transforms the final-state problem for \eqref{nls} into the initial-value problem
\begin{equation}\label{pcnls}
\begin{cases}
(i\partial_t+\Delta)v = \mu t^{\frac{dp}{2}-2}|v|^p v, \\
v(0) = \overline{\widehat{u_+}}\in L^2.
\end{cases}
\end{equation}
The scaling symmetry for \eqref{pcnls}, namely, $v(t,x)\mapsto \lambda^{d-\frac{2}{p}} v(\lambda^2 t, \lambda x)$ identifies \eqref{pcnls} as an \emph{$L^2$-supercritical} problem for $p<\frac{4}{d}$.  

\subsection*{Main Result} Inspired by recent works concerning almost sure well-posedness for supercritical problems (see e.g. 
\cite{BOP, Bourgain1, Bourgain2, BTT, BT, BT2, CO, Deng, LM, NORS, PRT, Tz} and the references therein), we consider the question of the `almost sure existence of wave operators' for the mass-subcritical NLS in the $L^2$-topology.  In particular, we will prove existence and uniqueness of solutions scattering to suitably randomized prescribed final states in $L^2$.  For the probabilistic aspects of this note, we largely follow the presentation of \cite{LM}. 

\begin{definition}[Randomization]\label{randomization} We perform a randomization in physical space.  This is similar to the randomization appearing in \cite{LM}, although there the randomization is in frequency space. 

 Let $\phi\in C_c^\infty(\R^d)$ be a smooth bump function such that $\phi=1$ for $|x|\leq 1$ and $\phi=0$ for $|x|>2$.  For $k\in\Z^d$, we set $\phi_k(x)=\phi(x-k)$ and define the partition of unity $\{\psi_k\}$ by
\[
\psi_k(x) = \frac{\phi_k(x)}{\sum_{\ell\in\Z^d} \phi_\ell(x)}. 
\]

Next, let $\{g_k\}_{k\in\Z^d}$ be a sequence of independent, mean-zero random variables on a probability space $(\Omega,\A, \P)$ with distributions $\mu_k$.  We assume that there exists $c>0$ so that
\begin{equation}\label{prob-condition}
\biggl| \int_\R e^{\gamma x}d\mu_k(x)\biggr| \leq e^{c\gamma^2}\qtq{for all}\gamma\in\R
\end{equation}
and all $k\in\Z^d$.  For simplicity, one can keep in mind the example of Gaussian random variables. 

For $f\in L^2$, we define the \emph{randomization} of $f$ by
\[
f^\omega(x) = \sum_{k\in\Z^d} g_k(\omega)\psi_k(x) f(x),
\]
which we understand as a limit in $L^2(\Omega;L_x^2(\R^d))$ of sums over $|k|\leq n$ as $n\to\infty$. 
\end{definition}

\begin{remark}\label{remark}This randomization procedure does not imply additional decay for $f^\omega$ in the sense of weighted bounds.  In particular, if there exists $c>0$ such that
\begin{equation}\nonumber
\sup_{k\in\Z^d}\mu_k([-c,c])<1,
\end{equation}
then the following holds for any $\eps>0$: if $|x|^\eps f\notin L^2$, then $|x|^\eps f^\omega\not\in L^2$ almost surely.  To prove this, one can mimic the proof of \cite[Lemma~B.1]{BT}, relying on the fact that
\begin{equation}\label{equiv}
\| |x|^\eps f\|_{L_x^2}^2 \sim \sum_{k\in\Z^d} \| |x|^\eps \psi_k f\|_{L_x^2}^2. 
\end{equation}
On the other hand, one can show that $\wh{f^\omega}$ almost surely enjoys additional regularity in the sense of higher $L^r$-norms.  See Section~\ref{S:remarks} for a further discussion. 
\end{remark}

The result of this note is the following `almost sure existence of wave operators' above the Strauss exponent. 

\begin{theorem}\label{T} Let $d\geq 1$, $\mu\in\{\pm1\}$, and 
\begin{equation}\label{p}
p_0(d) < p < \tfrac{4}{d},
\end{equation}
with $p_0(d)$ as in \eqref{Strauss}. Let $u_+\in L^2$ and define the randomization $u_+^\omega$ as in Definition~\ref{randomization}.  For almost every $\omega$, there exists a unique global solution $u$ to \eqref{nls} that scatters to $u_+^\omega$ in $L^2$, that is, 
\[
\lim_{t\to\infty} \|u(t) - e^{it\Delta}u_+^\omega\|_{L^2} = 0. 
\]
\end{theorem}

\begin{remark}\label{R:uniqueness} Like Nakanishi's result in \cite{Nakanishi}, Theorem~\ref{T} is valid in both the focusing and defocusing settings. Compared to his result, the novelty here is in the uniqueness of the solutions constructed (while still working at the level of $L^2$).  In addition, our result is valid in dimensions $d\in\{1,2\}$.  Note, however, that we were unable treat the whole range $\tfrac{2}{d}<p<\tfrac{4}{d}$.  
\end{remark}

\begin{remark} Uniqueness holds in the class of solutions in $C_t L_x^2$ that belong to $L_t^q L_x^r((T,\infty)\times\R^d)$ for some $T$, where $(q,r)$ is a certain `critical' exponent pair (see Section~\ref{S:FS}). 
\end{remark}

As in previous works treating supercritical problems via probabilistic techniques, an essential ingredient is the fact that solutions to the \emph{linear} Schr\"odinger equation with randomized initial data obey improved space-time estimates almost surely (see Proposition~\ref{prop:linear}).  Combining this fact with an `exotic' inhomogeneous Strichartz estimate, we are able to find a suitable space in which to run a simple contraction mapping argument (see Proposition~\ref{prop:wo}).  In fact, we are able to work with the space-time norms that one can access with the pseudoconformal energy estimate\footnote{This strategy for proving existence of wave operators appears already in the work of Kato \cite{Kato}.}.  We failed to find suitable spaces precisely when $p$ reaches the Strauss exponent; treating the range $\tfrac{2}{d}<p\leq \alpha_0(d)$ is an interesting open problem.

Whether or not one has deterministic existence of wave operators in $L^2$ (i.e. uniqueness in the result of Nakanishi \cite{Nakanishi}) remains an open question.  It is also natural to consider ill-posedness results in this setting, given the supercritical nature of the problem. 

By now there are many works addressing almost sure well-posedness in supercritical settings via probabilistic techniques.  There has also been some progress on almost sure \emph{scattering} in supercritical settings (e.g. \cite{BTT, DLM, PRT}).  The paper \cite{DLM} considers the defocusing energy-critical wave equation and uses a frequency space randomization.  The papers \cite{BTT, PRT} prove probabilistic global well-posedness for the NLS with a harmonic potential, using a randomization based on the Hermite functions.  By applying the lens transform, they deduce scattering results for the usual NLS.  In the mass-subcritical setting, they establish scattering with non-zero probability (for $L^2$ initial data) in dimensions $d\geq 2$.

The present work considers the final-state problem, which is essentially a local problem with data prescribed at $t=\infty$.  By employing a physical-space randomization, we are able to prove a scattering result for $L^2$ data that holds almost surely.  Furthermore, we are able to treat all dimensions $d\geq 1$, while the deterministic result of Nakanishi \cite{Nakanishi} only holds for $d\geq 3$ (with an extension to $H^1$ final states for $d=2$ \cite{HT}).

\subsection*{Outline of the paper} In Section~\ref{S:notation}, we introduce notation, collect estimates related to the linear Schr\"odinger equation, and introduce the requisite probabilistic results.  In Section~\ref{S:proof}, we prove Theorem~\ref{T}.  The main ingredients are improved space-time bounds for the linear Schr\"odinger equation with randomized data (Proposition~\ref{prop:linear}) and a deterministic existence of wave operators result (Proposition~\ref{prop:wo}).  In Section~\ref{S:remarks}, we conclude with some final remarks and prove almost sure additional `Fourier--Lebesgue regularity' for randomized final states (Proposition~\ref{flr}). 

\subsection*{Acknlowedgements} The author was supported by the NSF Postdoctoral Fellowship DMS-1400706.  I am grateful to K. Nakanishi, who pointed out that the arguments presented here could extend to dimensions $d\in\{1,2\}$. 

\section{Notation and useful lemmas}\label{S:notation}
We write $A\lesssim B$ to denote $A\leq CB$ for some $C>0$. We write $A\sim B$ to denote $A\lesssim B\lesssim A$.  We write $A\lesssim_\rho B$ to mean $A\leq CB$ for some $C=C(\rho)>0$.  We use the notation $L_t^q L_x^r(I\times\R^d)$ to denote space-time norms
\[
\|u\|_{L_t^q L_x^r(I\times\R^d)} = \biggl(\int_I\biggl(\int_{\R^d} |u(t,x)|^r\,dx\biggr)^{\frac{q}{r}}\,dt\biggr)^{\frac{1}{q}},
\]
with the usual adjustments when $q$ or $r$ is infinite.  

We define the \emph{scaling} associated to an exponent pair $(q,r)$ by
\[
s(q,r) = \tfrac{d}{2}-(\tfrac{2}{q}+\tfrac{d}{r}). 
\]
A space $L_t^q L_x^r$ is \emph{critical} for \eqref{nls} if $s(q,r)=s_c:=\tfrac{d}{2}-\tfrac{2}{p}$.  In this case, the $L_t^q L_x^r$-norm is invariant under the rescaling \eqref{scaling}. 

For an exponent $r\in[1,\infty]$, we write $r'\in[1,\infty]$ to denote the dual exponent, i.e. the solution to $\tfrac1r+\tfrac1{r'}=1$.

We write $\wh{f}$ for the Fourier transform of a function $f$.  

\subsection{The linear Schr\"odinger equation} Solutions to the linear Schr\"odinger equation are generated by $e^{it\Delta}$, where
\begin{equation}\label{propagator}
[e^{it\Delta}f](x) = (4\pi it)^{-\frac{d}{2}}\int_{\R^d} e^{i|x-y|^2/4t}f(y)\,dy. 
\end{equation}
Using this identity together with unitarity on $L^2$, one immediately deduces the following dispersive estimate by interpolation:
\begin{equation}\label{dispersive}
\| e^{it\Delta}\|_{L_x^{r'}\to L_x^r} \lesssim |t|^{-(\frac{d}{2}-\frac{d}{r})}\qtq{for} 2\leq r\leq \infty.
\end{equation}

Solutions to the nonlinear equation \eqref{nls} satisfy the \emph{Duhamel formula}
\[
u(t) = e^{i(t-t_0)\Delta}u(t_0) - i\mu \int_{t_0}^t e^{i(t-s)\Delta}\bigl(|u|^pu \bigr)(s)\,ds
\]
for any $t_0,t\in I$. 

We recall the standard Strichartz estimates.  We call a pair $(a,b)$ \emph{admissible} if $s(a,b)=0$ and $2\leq a\leq\infty$.  One also excludes the triple $(d,a,b)=(2,2,\infty)$ and requires $4\leq a\leq\infty$ in dimension $d=1$.  We call $(\alpha,\beta)$ \emph{dual admissible} if $(\alpha',\beta')$ is admissible.

\begin{proposition}[Strichartz estimates, \cite{GV, KT, Strichartz}]\label{prop:strichartz} Let $I$ be a time interval and $t_0\in \bar{I}$.  Then for any admissible $(a,b)$ and any dual admissible $(\alpha,\beta)$, we have
\begin{align}
\nonumber\|e^{it\Delta}\varphi\|_{C_t L_x^2 \cap L_t^a L_x^b(I\times\R^d)}& \lesssim \|\varphi\|_{L_x^2}, \\
\label{inhomogeneous}\biggl\| \int_{t_0}^t e^{i(t-s)\Delta}F(s)\,ds \biggr\|_{C_t L_x^2 \cap L_t^a L_x^b(I\times\R^d)}& \lesssim \|F\|_{L_t^\alpha L_x^\beta(I\times\R^d)}. 
\end{align}
\end{proposition}

The estimate \eqref{inhomogeneous} holds for more exponents than the ones given in Proposition~\ref{prop:strichartz}.  The sharp range of `exotic' inhomogeneous estimates was studied by Foschi \cite{Foschi}.  We record here one particular estimate that can be proven simply (see e.g. \cite[Chapter 2.4]{Cazenave}), which will be essential to our arguments below. 

\begin{proposition}[Inhomogeneous estimate]\label{prop:foschi} Let $0<p<\tfrac{4}{d-2}$ if $d\geq 3$ and $0<p<\infty$ if $d\in\{1,2\}$.  Let $I$ be a time interval and $t_0\in\bar{I}$. Let $r=p+2$ and suppose $1<q,\bar q<\infty$ satisfy
\begin{equation}\label{foschi-scale}
\tfrac{1}{q}+\tfrac{1}{\bar q} = \tfrac{d}{2}-\tfrac{d}{r}.
\end{equation}
Then
\[
\biggl\| \int_{t_0}^t e^{i(t-s)\Delta}F(s)\,ds\biggr\|_{L_t^q L_x^r(I\times\R^d)} \lesssim \|F\|_{L_t^{\bar q'} L_x^{r'}(I\times\R^d)}. 
\]
\end{proposition}

To prove this result, one uses \eqref{dispersive} and the Hardy--Littlewood--Sobolev inequality. 
\subsection{Function spaces}\label{S:FS}  We now introduce the specific exponents that we will use in the estimates below.  We first define
\[
r = p+2, \quad q = \tfrac{2p(p+2)}{4-p(d-2)}, \quad \bar q = \tfrac{2p(p+2)}{dp^2-(4-p(d-2))}.
\]
Then one can check $s(q,r)=s_c$ and that $(r,q,\bar q)$ satisfy the scaling relation \eqref{foschi-scale}. 

It is easy to verify that
\[
\max\{1,p\} < q < \infty
\]
for $p$ satisfying \eqref{p}; in fact this holds for a wider range of $p$ than the one appearing in \eqref{p}, for example $\tfrac{4}{d+2}<p<\tfrac{4}{d-2}$.  

One can also check that $\bar{q}>1$ for $p<\frac{4}{d-2}$, while $\bar q<\infty$ precisely when $p>p_0(d)$ (cf. \eqref{Strauss}).  This is one reason why we work above the Strauss exponent (see also Proposition~\ref{prop:linear}). 

In particular, we can apply Proposition~\ref{prop:foschi} with $(r,q,\bar{q})$.  As one can check that 
\[
r'=\tfrac{r}{p+1}\qtq{and}\bar{q}' = \tfrac{q}{p+1}, 
\]
we get the nonlinear estimate
\begin{equation}\label{nle}
\biggl\| \int_{t_0}^t e^{i(t-s)\Delta}\bigl(|u|^p u\bigr)(s)\,ds\biggr\|_{L_t^q L_x^r} \lesssim \| |u|^p u\|_{L_t^{\bar q'} L_x^{r'}} \lesssim \|u\|_{L_t^q L_x^r}^{p+1}.
\end{equation}

We next define an admissible and dual admissible pair.  For $d\geq 2$, we first take $a$ satisfying
\begin{equation}\label{a-conditions}
\max\{\tfrac12-\tfrac{p}{q}, 0\} < \tfrac{1}{a} < \min\{1-\tfrac{p}{q}, \tfrac12\}
\end{equation}
and choose $b$ so that $s(a,b)=0$; in particular, $(a,b)$ is an admissible pair.  For $d=1$, we instead impose
\begin{equation}\tag{\ref{a-conditions}}
\max\{\tfrac34-\tfrac{p}{q},0\} < \tfrac1a < \min\{1-\tfrac{p}{q}, \tfrac14\}.
\end{equation}
We now define $(\alpha,\beta)$ via
\[
\tfrac{1}{\alpha} = \tfrac{p}{q}+\tfrac{1}{a}\qtq{and}\tfrac{1}{\beta} = \tfrac{p}{r} + \tfrac{1}{b}.
\]
Then \eqref{a-conditions} implies $1<\alpha< 2$ in dimensions $d\geq 2$ and $1<\alpha<\frac43$ in $d=1$.  The scaling relations $s(q,r)=s_c$ and $s(a,b)=0$ then guarantee that $(\alpha,\beta)$ is a dual admissible pair.

We have the following nonlinear estimate via Proposition~\ref{prop:strichartz}:
\begin{equation}\label{nle2}
\biggl\| \int_{t_0}^t e^{i(t-s)\Delta}\bigl(|u|^p u\bigr)(s)\,ds\biggr\|_{C_t L_x^2\cap L_t^a L_x^b} \lesssim \||u|^p u\|_{L_t^\alpha L_x^\beta}\lesssim \|u\|_{L_t^q L_x^r}^p \|u\|_{L_t^a L_x^b}. 
\end{equation}

\subsection{Probabilistic results} We next import a few probabilistic results that will play a role in establishing improved space-time estimates for linear solutions with randomized data.  

The first result appears in \cite[Lemma~3.1]{BT}. 

\begin{proposition}[Large deviation estimate, \cite{BT}]\label{prop:ld} Let $\{\ell_k\}$ be a sequence of independent random variables with distributions $\{\mu_k\}$ on a probability space $(\Omega,\A,\P)$ satisfying  \eqref{prob-condition}. Then for $\alpha\geq 2$ and $\{c_k\}\in\ell^2$, 
\[
\biggl\| \sum_k c_k \ell_k(\omega)\biggr\|_{L_\omega^\alpha(\Omega)} \lesssim \sqrt{\alpha}\biggl( \sum_{k} |c_k|^2\biggr)^{\frac12}.
\]
\end{proposition}

The next lemma appears in \cite[Lemma~2.5]{LM}, which is in turn an adaptation of \cite[Lemma~4.5]{Tz}.

\begin{lemma}\label{L:prob} Suppose $F$ is a measurable function on a probability space $(\Omega,\A,\P)$. Suppose that there exist $C>0$, $A>0$, and $\alpha_0\geq 1$ such that for $\alpha\geq\alpha_0$, we have
\[
\| F\|_{L_\omega^\alpha(\Omega)} \leq C\sqrt{\alpha} A.
\]
Then there exists $C'=C'(C,\alpha_0)>0$ and $c=c(C,\alpha_0)>0$ such that
\[
\P(\omega\in\Omega:|F(\omega)|>\eta) \leq C' \exp\{-c\eta^2 A^{-2}\}.
\]
\end{lemma}

\section{Proof of Theorem~\ref{T}}\label{S:proof} 

This section contains the proof of Theorem~\ref{T}.  We first prove some improved space-time estimates for solutions to the linear Schr\"odinger equation with randomized data (Proposition~\ref{prop:linear}).  We then prove a deterministic existence of wave operators result (Proposition~\ref{prop:wo}).  These two results together will quickly imply Theorem~\ref{T}.

\subsection*{Space-time estimates with randomized data} Recall the exponents $(q,r)$ defined in Section~\ref{S:FS}.  We define
\[
\eps_0 := (\tfrac{d}{2}-\tfrac{d}{r}) - \tfrac{1}{q} = \tfrac{dp^2+p(d-2)-4}{2p(p+2)}
\]
and note that $\eps_0>0$ precisely when $p>p_0(d)$ (cf. \eqref{Strauss}).  This is another reason that we work above the Strauss exponent. 

\begin{proposition}\label{prop:linear} Let $T\geq 1$ and $u_+\in L^2$. Define $u_+^\omega$ as in Definition~\ref{randomization} and consider the set
\[
\Omega_{\eta,T} := \{\omega\in\Omega: \| e^{it\Delta}u_+^\omega\|_{L_t^q L_x^r((T,\infty)\times\R^d)}<\eta\}
\]
for some $\eta>0$.  There exists $C>0$ such that
\begin{equation}\label{prop:linear2}
\P(\Omega_{\eta,T}^c)\lesssim \exp\bigl\{-C \eta^2T^{2\eps_0} \|u_+\|_{L^2}^{-2}\bigr\}. 
\end{equation}
\end{proposition}

\begin{proof} Fix $\alpha>\max\{q,r\}$.  Changing the order of integration and using the large deviation estimate (Proposition~\ref{prop:ld}), the dispersive estimate \eqref{dispersive}, and H\"older's inequality, we estimate
\begin{align*}
\| e^{it\Delta} u_+^\omega\|_{L_\omega^\alpha(\Omega; L_t^q L_x^r((T,\infty)\times\R^d))} & \lesssim \biggl\|\,\biggl\| \sum_{k\in\Z^d} g_k(\omega)[e^{it\Delta}\psi_k u_+]\biggr\|_{L_\omega^\alpha}\biggr\|_{L_t^q L_x^r((T,\infty)\times\R^d)} \\
& \lesssim \sqrt{\alpha}\biggl\| \biggl(\sum_{k\in\Z^d} |e^{it\Delta}\psi_k u_+|^2 \biggr)^{\frac12}\biggr\|_{L_t^q L_x^r((T,\infty)\times\R^d)} \\
& \lesssim \sqrt{\alpha}\biggl\| \biggl(\sum_{k\in\Z^d}\| e^{it\Delta}\psi_k u_+\|_{L_x^r}^2\biggr)^{\frac12}\biggr\|_{L_t^q((T,\infty))} \\
& \lesssim \sqrt{\alpha}\| |t|^{-(\frac{d}{2}-\frac{d}{r})}\|_{L_t^q((T,\infty))} \biggl(\sum_{k\in\Z^d} \|\psi_k u_+\|_{L_x^{r'}}^2\biggr)^{\frac12} \\
& \lesssim \sqrt{\alpha}T^{-\eps_0}\biggl(\sum_{k\in\Z^d}\|\psi_k u_+\|_{L_x^2}^2\biggr)^{\frac12} \lesssim \sqrt{\alpha}T^{-\eps_0}\|u_+\|_{L_x^2}. 
\end{align*}
To pass from this estimate to \eqref{prop:linear2}, we now use Lemma~\ref{L:prob}.\end{proof}

\begin{remark} The use of H\"older's inequality in the argument above is akin to the `unit scale Bernstein estimate' in \cite{LM}. 
\end{remark}

In Section~\ref{S:remarks}, we discuss some conditions on $u_+$ that imply $L_t^q L_x^r$ bounds for $e^{it\Delta}u_+$ in the deterministic setting.

\subsection*{Final-state problem} We next prove a deterministic result concerning existence of wave operators.  We recall the exponents defined in Section~\ref{S:FS}. 

\begin{proposition}\label{prop:wo} Let $d\geq 3$ and $p_0(d)<p<\tfrac{4}{d}$, with $p_0(d)$ as in \eqref{Strauss}.  There exists $\eta_0>0$ such that the following holds: If $\varphi\in L^2$ and $T$  is such that
\begin{equation}\label{smallness}
\| e^{it\Delta}\varphi\|_{L_t^q L_x^r((T,\infty)\times\R^d)} < \eta\qtq{for some}0<\eta<\eta_0,
\end{equation}
then there exists a unique global solution $u$ to \eqref{nls} satisfying
\begin{equation}\label{scatter}
\lim_{t\to\infty}\|u(t)-e^{it\Delta}\varphi\|_{L^2}=0. 
\end{equation}
\end{proposition}

\begin{proof} We define
\begin{align*}
X = \bigl\{&u\in C_t L_x^2 \cap L_t^a L_x^b \cap L_t^q L_x^r((T,\infty)\times\R^d):\\  
&\|u\|_{C_t L_x^2\cap L_t^a L_x^b((T,\infty)\times\R^d)}\leq 2C\|\varphi\|_{L^2},\\ 
&\|u\|_{L_t^q L_x^r((T,\infty)\times\R^d)} \leq C\eta \bigr\},
\end{align*}
which is complete with respect to
\[
d(u,v) = \|u-v\|_{L_t^a L_x^b((T,\infty)\times\R^d)}. 
\]
Here $C$ encodes various constants appearing in the estimates below.  We will show that for $\eta$ sufficiently small, 
\[
[\Phi u](t) := e^{it\Delta}\varphi +i\mu\int_t^\infty e^{i(t-s)\Delta}\bigl(|u|^p u\bigr)(s)\,ds
\]
is a contraction from $X$ to $X$ with respect to $d(\cdot,\cdot)$.  In the following, we take all space-time norms over $(T,\infty)\times\R^d$.

First note that by Strichartz (Proposition~\ref{prop:strichartz}) and \eqref{smallness},
\[
\| e^{it\Delta}\varphi\|_{L_t^a L_x^b} \lesssim \|\varphi\|_{L^2},\quad \|e^{it\Delta}\varphi\|_{L_t^q L_x^r}<\eta. 
\] 
Now let $u\in X$.  Estimating as in \eqref{nle2} via Proposition~\ref{prop:strichartz}, we have
\begin{equation}\label{est-1}
\biggl\| \int_t^\infty e^{i(t-s)\Delta}\bigl(|u|^p u\bigr)(s)\,ds\biggr\|_{C_t L_x^2 \cap L_t^a L_x^b} \lesssim\|u\|^p_{L_t^q L_x^r} \|u\|_{L_t^a L_x^b} \lesssim \eta^p\|\varphi\|_{L^2}. 
\end{equation}
Similarly,  estimating as in \eqref{nle} via Proposition~\ref{prop:foschi}, 
\begin{align*}
\biggl\| \int_t^\infty e^{i(t-s)\Delta}\bigl(|u|^p u\bigr)(s)\,ds\biggr\|_{L_t^q L_x^r} \lesssim \|u\|_{L_t^q L_x^r}^{p+1}  \lesssim \eta^{p+1}. 
\end{align*}
Thus, for $\eta$ sufficiently small, $\Phi:X\to X$.  Furthermore, estimating similarly to \eqref{est-1} yields
\[
d(\Phi u,\Phi v) \lesssim \eta^p d(u,v),
\]
so that $\Phi$ is a contraction for $\eta$ small enough. 

Consequently, $\Phi$ has a unique fixed point $u\in X$.  It is now standard to verify that $u$ is a solution to \eqref{nls} on $(T,\infty)$ that satisfies \eqref{scatter}.  Furthermore, as described in the introduction, $u$ is automatically global-in-time.\end{proof}

We are now ready to complete the proof of Theorem~\ref{T}.

\begin{proof}[Proof of Theorem~\ref{T}] Fix $u_+\in L^2$ and define the randomization $u_+^\omega$ as in Definition~\ref{randomization}.  Now choose $\eta_0$ as in Proposition~\ref{prop:wo} and fix $0<\eta<\eta_0$.  Define the sets $\Omega_{\eta,T}$ as in Proposition~\ref{prop:linear}. Then, using Proposition~\ref{prop:linear}, we find that the set $\Omega_\eta := \cup_{T=1}^\infty \Omega_{\eta,T}$ has $\P(\Omega_\eta)=1$.  

Furthermore, for each $\omega\in\Omega_\eta$, we may apply Proposition~\ref{prop:wo} with $\varphi= u_+^\omega$ and $T$ chosen so that $\omega\in \Omega_{\eta,T}$.  In particular, we find that there exists a unique global solution $u$ that scatters to $u_+^\omega$.  This completes the proof of Theorem~\ref{T}.
\end{proof}

\section{Remarks}\label{S:remarks}

Using Proposition~\ref{prop:wo}, we can conclude existence of wave operators (for a restricted range of powers $p$) if $u_+\in L^2$ satisfies the additional condition
\begin{equation}\label{fourier}
\wh{u_+} \in L_\xi^{\rho_0}(\R^d),\qtq{where}\rho_0:=\tfrac{dp}{dp-2}. 
\end{equation}
Indeed, this is a consequence of the following Strichartz estimate (see \cite[Proposition~2.4]{Masaki}) and the montone convergence theorem:

\begin{proposition}[Fourier--Lebesgue Strichartz estimate\cite{Masaki}]\label{prop:masaki} Let $d\geq 3$ and take $(q,r)$ as in Section~\ref{S:FS}.  Then
\[
\|e^{it\Delta}\varphi\|_{L_t^q L_x^r(\R\times\R^d)} \lesssim \|\wh{\varphi}\|_{L_\xi^{\rho_0}(\R^d)},
\]
provided $p$ is chosen so that the following constraints are satisfied:
\begin{equation}\label{hat-restraint}
\tfrac{1}{q} \leq \tfrac{d}{d-2}\tfrac{1}{r}, \quad \tfrac{1}{q} < \tfrac{d+1}{2(d+3)}-\tfrac{d+1}{2}(\tfrac1r - \tfrac{d+1}{2(d+3)}).
\end{equation}
\end{proposition}
This estimate is proven using Fourier restriction estimates, which are ultimately the source of the constraints \eqref{hat-restraint}.  Note that \eqref{hat-restraint} is always satisfied in a neighborhood of $p=\frac4d$, although not typically in our full range of interest $p_0(d)<p<\tfrac{4}{d}$.

One also has existence of wave operators if $u_+\in L^2$ satisfies the weighted assumption
\begin{equation}\label{weighted}
|x|^{|s_c|}u_+\in L^2(\R^d),\quad s_c = \tfrac{d}{2}-\tfrac{2}{p},
\end{equation}
as is already well-known in the literature (see e.g. \cite{Mas2, NakanishiA..., Nakanishi}).  In this case, one does not need to use exotic inhomogeneous estimates; one can instead work with the quantity $e^{it\Delta}|x|^{|s_c|}e^{-it\Delta}u$ in admissible (Lorentz-modified) Strichartz spaces.  

The weighted assumption in \eqref{weighted} is a strictly stronger assumption than the `Fourier--Lebesgue regularity' in \eqref{fourier}; indeed, by Sobolev embedding and Plancherel, one has
\begin{equation}\label{embedding}
\| \wh{u_+}\|_{L_\xi^{\rho_0}} \lesssim \| |\nabla|^{|s_c|}\wh{u_+}\|_{L_\xi^2} \sim \||x|^{|s_c|}u_+\|_{L_x^2}. 
\end{equation}
We failed to find a simple proof of boundedness of $e^{it\Delta}u_+$ in the particular space $L_t^q L_x^r$ under the assumption \eqref{weighted} in the range of interest $p_0(d)<p<\tfrac{4}{d}$.  Of course, whenever \eqref{hat-restraint} is satisfied, boundedness follows from Proposition~\ref{prop:masaki} and \eqref{embedding}.

Finally, returning to the discussion in Remark~\ref{remark}, we note that while the randomization in Definition~\ref{randomization} does not yield weighted bounds, it does imply additional Fourier--Lebesgue regularity almost surely.
\begin{proposition}[Fourier--Lebesgue regularity]\label{flr}  Let $u_+\in L^2$ and $\rho\in(2,\infty)$. Define $u_+^\omega$ as in Definition~\ref{randomization} and consider the set
\[
\Omega_{M,\rho} := \{\omega\in\Omega: \| \wh{u_+^\omega}\|_{L_\xi^\rho} \leq M\}
\]
for $M>0$.  There exists $c=c(\rho)>0$ such that
\[
\P(\Omega_{M,\rho}^c) \lesssim_\rho \exp\{-c M^2\|u_+\|_{L^2}^{-2}\}. 
\]
Consequently, for any $\frac{2}{d}<p<\tfrac{4}{d}$, we have $\wh{u_+^\omega}\in L_\xi^{\rho_0}$ almost surely.
\end{proposition}
\begin{proof}  The argument is similar to the proof of Proposition~\ref{prop:linear}. We fix $\alpha>\rho$.  We change the order of integration and use the large deviation estimate (Proposition~\ref{prop:ld}), the Hausdorff--Young inequality, and H\"older's inequality to estimate
\begin{align*}
\| \wh{u_+^\omega}\|_{L_\omega^\alpha L_\xi^\rho}     \lesssim \biggl\| \,\biggl\| \sum_{k\in\Z^d} g_k(\omega)\wh{\psi_k u_+}\biggr\|_{L_\omega^\alpha} \biggr\|_{L_\xi^\rho} & \lesssim \sqrt{\alpha} \biggl\|\biggl(\sum_{k\in\Z^d}\bigl| \wh{\psi_k u_+}\bigr|^2 \biggr)^{\frac12}\biggr\|_{L_\xi^\rho} \\
& \lesssim \sqrt{\alpha}\biggl(\sum_{k\in\Z^d} \| \wh{\psi_k u_+}\|_{L_\xi^\rho}^2 \biggr)^{\frac12} \\
& \lesssim  \sqrt{\alpha}\biggl(\sum_{k\in\Z^d} \|\psi_k u_+\|_{L_x^{\rho'}}^2\biggr)^{\frac12} \\
& \lesssim  \sqrt{\alpha}\biggl(\sum_{k\in\Z^d}\|\psi_k u_+\|_{L_x^2}^2\biggr)^{\frac12}\lesssim  \sqrt{\alpha}\|u_+\|_{L_x^2}. 
\end{align*}
Thus the desired estimate follows from Lemma~\ref{L:prob}.  

To get the final conclusion we note that $\Omega_0:=\cup_{M=1}^\infty\Omega_{M,\rho_0}$ satisfies $\P(\Omega_0)=1$ and $\wh{u_+^\omega}\in L_\xi^{\rho_0}$ for all $\omega\in\Omega_0$. 
\end{proof}

\end{document}